\documentclass[12pt]{amsart}

\usepackage{}
 \usepackage{mathrsfs}
\usepackage{amsfonts, amssymb,amsmath, amsthm, amsxtra, latexsym, amscd}
\usepackage[latin1]{inputenc}

\theoremstyle{plain}
\newtheorem*{thm*}{Theorem}

\newtheorem{thm}{Theorem}[section]
\newtheorem{cor}[thm]{Corollary}
\newtheorem{lem}[thm]{Lemma}
\newtheorem{prop}[thm]{Proposition}

\numberwithin{equation}{section}

\bigskip

\def\refn#1.#2{\expandafter\def\csname#1\endcsname{[#2]}}
\def\refnr#1.{\csname#1\endcsname}

\begin{document}

\baselineskip  1.25pc

\title[reducing subspace for multiplication operators of the Bergman space]
{Reducing subspaces for analytic multipliers of the Bergman space
}
\author{Ronald G. Douglas, Mihai Putinar and Kai Wang }
\address{  Department of Mathematics,
Texas A\&M University, College Station, TX 77843, USA}
\email{rdouglas@math.tamu.edu }
\address{Department of Mathematics,
University of California at Santa Barbara,
Santa Barbara, CA 93106, USA}
\email{mputinar@math.ucsb.edu }
\address{School of Mathematical Sciences,
Fudan University, Shanghai, 200433, P. R. China}
\email{kwang@fudan.edu.cn}

 \subjclass[2010]{ 47B35;  30D50; 46E20}
\keywords{reducing subspace, Bergman space,  finite Blaschke product}
\thanks{The second author was   supported by NSF (DMS 1001071) and the workshop in Analysis and Probability at Texas A\&M University.          The third author was  supported by
  NSFC (10731020,10801028),   the Department of Mathematics  at Texas A\&M University and Laboratory of Mathematics for Nonlinear Science at Fudan University.
}
\begin{abstract}
We answer affirmatively   the problem left open in \cite{DSZ,GSZZ} and prove that for a finite Blaschke product $\phi$, the
minimal reducing subspaces of the Bergman space multiplier $M_\phi$ are pairwise orthogonal  and their number is equal to the number $q$ of connected components of the Riemann surface of $\phi^{-1}\circ \phi$. In particular,
the double commutant $\{M_\phi,M_\phi^\ast\}'$ is abelian of dimension $q$. An analytic/arithmetic description of the minimal reducing subspaces of $M_\phi$ is also provided, along with a list of all possible cases in degree of $\phi$  equal to eight.
\end{abstract}
\maketitle

\section{Introduction}

The aim of the present note is to classify the reducing subspaces of analytic
Toeplitz operators with a rational, inner symbol acting on the Bergman
space of the unit disk. While a similar study in the case of the Hardy space was completed a long time
ago  (see \cite{Co,To1,To2}), investigation of the Bergman space
setting was started only a few years ago. Not
surprisingly, the structure and relative position of these reducing
subspaces in the Bergman space reveal a rich geometric (Riemann surface)
picture directly dependent on the rational symbol of the Toeplitz
operator.

We start by recalling a few basic facts and some terminology. The Bergman space $L^2_a(\mathbb{D})$ is the space of holomorphic functions on $\mathbb{D}$ which are square-integrable with respect to the Lebesgue measure $dm$ on $\mathbb{D}$.
For a bounded holomorphic  function $\phi$ on the unit disk, the multiplication operator, $M_\phi:L^2_a(\mathbb{D})\to L^2_a(\mathbb{D})$, is defined by
 $$M_\phi(h)=\phi h,\,\,h\in L^2_a(\mathbb{D}). $$
 The    Toeplitz operator $T_\phi$ on $L^2_a(\mathbb{D})$ with symbol $\phi\in L^\infty(\mathbb{D})$ acts as
         $$T_\phi(h)=P(\phi h), \,\,h\in  L_{a}^2,$$
         where $P$ is the orthogonal projection from  $L^2(\mathbb{D})$ to $L^2_a(\mathbb{D}).$  Note that $T_\phi=M_\phi$ whenever $\phi$ is holomorphic.

 An invariant subspace $\mathcal{M}$ for $M_\phi$ is a closed subspace of $L_a^2(\mathbb{D})$ satisfying $ \phi \mathcal{M}\subseteq \mathcal{M}$. If, in addition, $M_\phi^\ast \mathcal{M}\subseteq \mathcal{M}$, we call  $\mathcal{M}$  a reducing subspace of $M_\phi$. We say $\mathcal{M}$ is a minimal reducing subspace if there is no nontrivial reducing subspace for $M_\phi$ contained in $\mathcal{M}$. The study of invariant subspaces and reducing subspaces for various classes of linear operators  has inspired much deep research  and prompted many interesting  problems.   Even for the  multiplication operator  $M_z$, the lattice of invariant subspaces of $L^2_a(\mathbb{D})$ is huge and its order structure remains a mystery.  Progress in understanding the lattice of reducing subspaces of $M_\phi$ was only recently made, and only in the case of  inner function symbols \cite{DSZ,GH,GH2,GH3,GSZZ,SW,SZZ, Zhu}.

 Let $ \{\mathcal{M}_\phi \}'=\{X\in \mathscr{L}(L_a^2(\mathbb{D})) :\, M_\phi X=X M_\phi\}$ be the commutant algebra of $M_\phi$. The problem of classifying the reducing subspaces of $M_\phi$ is equivalent to finding the projections in $ \{\mathcal{M}_\phi \}'$. This classification problem in the case of the Hardy space was the motivation of the highly original works by Thomson and Cowen (see  \cite{Co,To1,To2}). They used the Riemann surface of $\phi^{-1}\circ \phi$ as a basis for the description of the commutant of $M_\phi$ acting on the Hardy space. Notable for our study is that inner function symbols played a dominant  role in their studies. In complete analogy, in the Bergman space $L_a^2(\mathbb{D})$ framework,  one can use essentially the same proof to show that for a "nice" analytic function  $f$, there exists a finite Blaschke product $\phi$ such that $\{M_f\}'=\{M_\phi\}'$. Therefore, the structure of the reducing subspaces of the multiplier $M_f$
 on the Bergman space of the disk is the same as that for $M_\phi$.

  Zhu  showed in \cite{Zhu} that for each Blaschke product of order $2$,  there exist exactly  $2$ different minimal reducing subspaces of $M_\phi$. This result also appeared in \cite{SW}. Zhu also  conjectured in \cite{Zhu} that  $M_\phi$ has exactly  $n$ distinct  minimal reducing subspaces for a Blaschke product $\phi$ of order $n$. The results in \cite{GSZZ} disproved Zhu's conjecture, and the authors raised  a modification   in which $M_\phi$ was conjecture to have at most $n$ distinct minimal reducing subspaces for  a Blaschke product $\phi$ of order $n$. Some partial results on this conjecture  were obtained in \cite{GH,GSZZ,SZZ}.  These
  authors proved the finiteness result in   case  $n\leq 6 $, each using a different method. A notable result for the general case    \cite{GSZZ} is that there always exists  a nontrivial minimal reducing subspace $\mathcal{M}$, named the "distinguish subspace", on which the action  of $M_\phi$ is unitarily  equivalent to the action  of $M_z$ on the Bergman space $L^2_a(\mathbb{D})$. Guo and Huang    also revealed in \cite{GH2} an interesting connection between
  the structure of the lattice of reducing subspaces of $M_\phi$  and an  isomorphism problem in  abstract  von Neumann algebras. The general case was recently studied by   the first author, Sun and Zheng \cite{DSZ} using a systematic analysis of the local inverses of the ramified finite fibration $\phi^{-1}\circ \phi$ over the disk. They proved that the linear dimension of the commutant $\mathcal{A}_\phi=\{\mathcal{M}_\phi, \mathcal{M}_\phi^\ast\}'$ is finite and equal to the number of connected components of the Riemann surface of $\phi^{-1}\circ \phi$. As a consequence, one finds that the number of pairwise orthogonal reducing subspaces of $M_\phi$ is finite. In \cite{DSZ} the  authors raised the following question, whose validity they have established in degree $n \leq 8$.\\

 { \bf{Conjecture.}} For a Blaschke product $\phi$ of  finite order,  the double commutant algebra   $\mathcal{A}_\phi$  is  abelian.\\

Several notable corollaries would  follow  once  one proves the conjecture. For instance, the
commutativity of the algebra $A_\phi$ implies that, for every finite Blaschke product
$\phi$, the minimal reducing subspaces of $M_\phi$ are mutually orthogonal; in addition,
their number is equal to the number $q$ of connected components of the Riemann surface of $\phi^{-1}\circ \phi$.

 The main result of this paper  (contained in Section 2) offers an
affirmative answer to the
above conjecture.
\begin{thm}
Let $\phi$ be a finite Blaschke product of order $n$. Then the von Neumann algebra $\mathcal{A}_\phi=\{M_\phi,M_\phi^\ast\}'$ is commutative of dimension $q$, and hence $\mathcal{A}_\phi\cong\underbrace{\mathbb{C}\oplus\cdots\oplus \mathbb{C}}_q,$ where $q$ is the number   of connected components of the Riemann surface of $\phi^{-1}\circ \phi$.
\end{thm}
The key observation for the proof is that there is an invertible  holomorphic function $u$    such that $\phi=u^n$ on  $\Omega$, where $\Omega$ is a domain   in $\mathbb{D}$ including an annulus of all points  sufficiently  close  to the boundary  $\mathbb{T}$. This implies that local inverses for $\phi^{-1}\circ \phi$ commute  under composition on $\Omega$.

It also allows us to provide an indirect description of  the reducing subspaces.  For convenience, we introduce some additional notations. Following \cite{DSZ}, there is a partition    $\{G_1,\cdots,G_q\}$    of the local inverses for $\phi^{-1}\circ \phi$. We now define a dual partition as follows.  For two integers $0\leq j_1,j_2\leq n-1,$  write
  $j_1\sim j_2 $ if
  \begin{equation}\label{dual}\sum_{\rho_k\in G_i}\zeta^{k\,j_1\,}=\sum_{\rho_k\in G_i}\zeta^{k\,j_2 }  \text{ for any  } 1\leq i\leq q.\end{equation}
Observing that $\sim$ is an equivalence relation, we partition the set $\{0,1,\cdots,n-1\}$ into equivalence classes
  $\{G'_1,\cdots,G'_p\}.$ Some information on the  Riemann surface of $\phi^{-1}\circ \phi$ is given by  the  following  corollary in Section 3.
   \begin{cor}
 The number of  components in  the dual partition is also equal to $q$, the number of connected components of the Riemann surface for $\phi^{-1}\circ \phi$.
 \end{cor}
  Furthermore, we obtain the following characterization for the minimal reducing subspace of   automorphic type  in Section 3. Here ${\mathcal O}(\mathbb{D})$ denotes the space of holomorphic functions on $  \mathbb{D} $.
    \begin{thm}
 Let $\phi$ be a finite Blaschke product  and $\{G'_1,\cdots,G'_q\}$ be the dual partition for $\phi$. Then the multiplication operator $M_\phi$   has exactly $q$ nontrivial minimal  reducing subspaces $\{\mathcal{M}_1,\cdots,\mathcal{M}_q\},$ and for any $1\leq j\leq q$
 $$ \mathcal{M}_j=\{f\in {\mathcal O}(\mathbb{D}): f|_\Omega \in  \mathcal{L}_j^\Omega\},$$
 where $\mathcal{L}^\Omega_j$ is a subspace of $L^2(\Omega) $ with the orthogonal basis  $ \{u^i u':   i+1(mod\,n) \in G'_j \}.$
 \end{thm}
 \noindent Note   the $\mathcal{M}_{n-1}$ coincides with the distinguish reducing subspace for $M_\phi$ shown to exist  in \cite{GSZZ}.
This latter theorem provides a possible way to calculate the reducing subspace if one knows the partition of the family of local inverses. The above corollary hints that
the possible       partitions  are very restricted.

 Finally, in Section 4 we list some algebraic conditions for  the partitions, which offer an arithmetic path towards the classification
of   finite Blaschke products. The idea is displayed by the classification  for the Blaschke products of order $8$. In a similar way one can also explain   the classifications of the Blaschke products of order $3$ or $4$ in \cite{GSZZ,SZZ},  which  have been established    by identifying the Bergman space of the disk with the restriction of the Hardy space of the bidisk to the diagonal. We point out that these results and examples provide some very detailed information about the branch covering space defined by a finite Blaschke product.

  \section{The double commutant algebra is  abelian}
%
%
  The notation below is borrowed from \cite{DSZ}.  Accordingly, throughout this article  $\phi$ is a finite Blaschke product having $n$ zeros taking  multiplicity into account. The finite set $E'= \phi^{-1}(\phi(\{\beta\in\mathbb{D}:\phi'(\beta)=0\})  ) $ denotes the branch points of $\phi$,  $E=\mathbb{D}\backslash E'$ is its complement  in $\mathbb{D}$ and let $\Gamma$ be a choice of curves passing through all  points of $ E'$ and a fixed point on the unit circle $\beta_0$ such that $\mathbb{D}\backslash\Gamma$ is a simply connected region contained in $E$. Indeed, to be precise, one can construct  $\Gamma$ as follows: order $E'$ as $\{\beta_1,\beta_2,\cdots,\beta_s\}$ such that $k\leq j$ iff $Re \beta_k \leq Re\beta_j$ or  $Re \beta_k = Re\beta_j$ and $Im \beta_k \leq  Im\beta_j$, and set $\beta_0=Re \beta_1+i\,\sqrt{1-(Re \beta_1)^2} $. Letting $\Gamma_k, \,\,0\leq k\leq s-1$ be the line segment between  $\beta_k$ and  $\beta_{k+1}$, we define \begin{equation}\label{Gamma1}\Gamma=\cup_{0\leq k\leq s-1} \Gamma_k .\end{equation}

    By an observation made in \cite{DSZ}, the family of analytic local inverses $\{\rho_0,\cdots,\rho_{n-1}\}$ for  $\phi^{-1}\circ \phi$ is well defined on $\mathbb{D}\backslash\Gamma$. That is, each $\rho_j$ is a holomorphic function  on $\mathbb{D}\backslash\Gamma$ which  satisfies $\phi(\,\rho_j(z))=\phi(z)$ for $z\in \mathbb{D}\backslash\Gamma$. We define the equivalence relation on the set of local inverse so that  $\rho_i \sim \rho_j$ if there exists an arc
  $\gamma$ in $E$ such that $\rho_i$ and $\rho_j$ are   analytic   continuations of each other  along $\gamma$. The resulting equivalence classes are denoted $\{G_{ 1},\cdots, G_{ q }\}$. For each $G_{ k}, 1\leq k\leq q $, define the   map $\mathcal{E}_k$:
  $$(\mathcal{E}_k f) (z)=\sum_{\rho\in G_{ k} }  f(\rho(z))\rho\,'(z),\,\,f\,\, holomorphic \,\,on \,\,\mathbb{D}\backslash\Gamma,\, z\in \mathbb{D}\backslash\Gamma.$$
 The  central result in  \cite{DSZ}  asserts that the operators    $\{\mathcal{E}_1,\cdots,\mathcal{E}_{q }\}$ can naturally  be extended to  bounded operators on the Bergman space $L^2_a(\mathbb{D})$ which are  linearly independent, and the double commutant algebra $\mathcal{A}_\phi$ is linearly generated by these operators; that is,
 $$ \mathcal{A}_\phi= \{M_\phi,M_\phi^\ast\}'={\text{span}} \{\mathcal{E}_1,\cdots,\mathcal{E}_{q }\}.$$ In this section we prove that the von Neumann algebra  $\mathcal{A}_\phi$ is commutative.

To accomplish this,  we extend    the given family of analytic local inverses on $\mathbb{D}\backslash \Gamma$  to a larger  region and  prove  that they commute under composition  near the boundary of $\mathbb{D}$. The key observation for the proof of the following lemma is that $ \sqrt[n]{(z-a_1)\cdots (z-a_n)}$ is a single-valued holomorphic function    on $\mathbb{C}\backslash L$, where $L$ is a curve drawn through the zero set $\{a_1,a_2,\cdots,a_n\}$. One can construct an $L$ and verify the above assertion  as follows. Notice that
 $ \sqrt[n]{ z+1 }$ is holomorphic outside any smooth simply curve  connecting $-1$ and $\infty$. By changing variables, we have for each $2\leq i\leq n$ that
 $$ \sqrt[n]{\frac{ z-a_i }{z-a_1}}=\sqrt[n]{\frac{ a_1-a_i }{z-a_1}+1}$$
 is holomorphic outside the line segment  connecting $a_1$ and $a_i$. Therefore,
 $$\sqrt[n]{(z-a_1)\cdots (z-a_n)}=(z-a_1) \sqrt[n]{\frac{ z-a_2 }{z-a_1}}\cdots \sqrt[n]{\frac{ z-a_n }{z-a_1}}$$
is holomorphic outside the arc which consists of the line segments connecting $a_1$ and $a_i$ for $2\leq i\leq n$. We refer the interested  reader  to \cite[Section 55]{Mar}   for a more careful argument.

Hereafter, let us set  $A_r=\{z\in\mathbb{C}: r<|z|<1\}$ for any $0<r<1$, and let $\zeta=e^{\frac{2  i  \pi}{n}}$ be a primitive $n$-th root of unity.

  \begin{lem}\label{u}
  For   a finite Blaschke product $\phi$ of order $n$, there exists   a holomorphic  function $u$ on a neighborhood of $\overline{\mathbb{D}}\backslash L$ such that $\phi=u^n$, where $L$ is an arc inside $\mathbb{D}$ containing the zero set of $\phi$. Moreover, there exists $0<r<1$ such that $\overline{A_r}$ is contained in the image of $u$ and   $u: u^{-1}(\overline{A_r}) \to \overline{A_r}$ is   invertible.
  \end{lem}
  \begin{proof} Suppose $a_1,\cdots, a_n$ are the zeros of $\phi$ in  $\mathbb{D}$ (taking multiplicity into account). Choose an analytic branch for $w=\sqrt[n]{z}$. By \cite[Section 55, p221]{Mar}, $w=\sqrt[n]{(z-a_1)\cdots (z-a_n)}$ is a single-valued holomorphic function    on $\mathbb{C}\backslash L$, where $L$ is a curve drawn through the zero set.  If we set
  $$u(z)=\frac{\sqrt[n]{(z-a_1)\cdots (z-a_n)}}{\sqrt[n]{(1-\overline{a_1} z)\cdots (1-\overline{a_n} z)}},$$
  then $u(z)$ is holomorphic  on a neighborhood of  $\overline{\mathbb{D}}\backslash L$ and $u^n=\phi$.

 Additionally, one sees   that $|u|^n=|\phi|$ on $\overline{\mathbb{D}}\backslash L$ and hence   $u(\mathbb{T})\subseteq\mathbb{T}$.  We claim that $u(\mathbb{T})=\mathbb{T}$. Indeed, if $u(\mathbb{T})\neq \mathbb{T}$, then $u:\mathbb{T}\to \mathbb{T}$ is homotopic to a constant map on
   $\mathbb{T}.$ That is, there exists $u(\theta,t)\in C(\mathbb{T}\times [0,1], \mathbb{T})$ such that $u(\theta,0)=u(\theta)$ and  $u(\theta,1)=1$.  This implies that $\phi=u^n:\mathbb{T}\to \mathbb{T}$ is also homotopic to the constant map by the path $t\to u^n(\cdot,t)$. If we extend each $u(\cdot,t)$ to be a continuous function $\widetilde{u}(\cdot,t)$ on $\overline{\mathbb{D}}$, then by   \cite[Theorem 1]{Cob} each Toeplitz operator $T_{ \widetilde{u}^n(\cdot,t) }$ is Fredholm. Furthermore, using  \cite[Theorem 1]{Cob} one sees that $t\to \mathrm{Ind} (T_{ \widetilde{u}^n(\cdot,t) })$ is a continuous  map from $[0,1]$ to $\mathbb{Z}$. This implies that it is a constant map, which leads to a  contradiction since  $-n=\mathrm{Ind}( M_{\phi})=\mathrm{Ind} (T_{ \widetilde{u}^n(\cdot,0) })=\mathrm{Ind} (T_{ \widetilde{u}^n(\cdot,1) })=\mathrm{Ind}( M_{1})=0. $  Therefore, we have that      $u(\mathbb{T})=\mathbb{T}$.

   By the open mapping theorem, the image of $u$ is an open subset of $\mathbb{C}$ including $\mathbb{T}$. Therefore, there exists $0<r <1$ such that $\overline{ A_{r }}  \subseteq u(\overline {\mathbb{D}}\backslash L) $.  Now we only need to prove that the map $u: u^{-1}(\overline{{A_{r}}}) \to \overline{A_{r}}$ is injective. In fact, for any $w\in \overline A_{r}$, since $\phi(u^{-1}(\zeta^k  w))= w^n$ for $0\leq k\leq n-1 $, we have that
   $$ \bigcup_{0\leq k\leq n-1} u^{-1}(\{\zeta^k w\}) \subseteq \phi^{-1}(\{w^n\}).$$
    Remarking  that the set $\phi^{-1}(\{w^n\})$ includes at most $n$ points and each set $u^{-1}(\{\zeta^k w\})$ is
    nonempty, one sees that each   $u^{-1}(\{\zeta^k w\})$ is  a singleton. This means that  $u$ is one to one  on $u^{-1}(\overline{{A_{r}}})$.  Therefore, $u: u^{-1}( \overline{A_{r }})\to \overline{A_{r}} $ is invertible, completing the proof.
  \end{proof}

The above lemma allows us to  extend
local inverses   as follows. Hereafter,  we denote $\Omega=u^{-1}(A_r)$, where $A_r$ is the annuals appearing in Lemma 2.1.   On the connected domain $ \Omega$, define $\widetilde{\rho}_k(z)=u^{-1} (\zeta^k  u(z))$ for each $0\leq k\leq n-1$. Note that $\widetilde{\rho}_k$ is holomorphic and $\phi(\widetilde{\rho}_k(z))=\phi(z)$ for $z\in \Omega$. This means that $\{\widetilde{\rho}_k\}_k$ is also the family of local inverses on $\Omega$ for $\phi^{-1}\circ \phi$. It follows that $\rho_k=\widetilde{\rho}_{i_k}$ for some $i_k$ on $\Omega\bigcap [\mathbb{D}\backslash \Gamma]$. Matching  the maps $\widetilde{\rho}_{i_k}$ and $\rho_k$, respectively,  we obtain the family of  local inverses on a  larger  domain $\Omega\bigcup [\mathbb{D}\backslash \Gamma]$. Furthermore, we can prove the following lemma.

\begin{lem}\label{loc}
For a finite Blaschke product $\phi$, there exists a family of  local inverses   for $\phi^{-1}\circ \phi$ on the domain $\mathbb{D}\backslash \Gamma'$, where $\Gamma'=\cup_{1\leq k\leq s-1} \Gamma_i$ is a proper subset of $\Gamma$ appearing in $(\ref{Gamma1})$,   which just consists of  the set of line segments passing  through  all critical points $E'$ of $\phi$.
\end{lem}
\begin{proof}
It suffices to show that the family of  local inverses $\{\rho_0,\rho_1,\cdots,\rho_{n-1}\}$ can be analytically continued  across   the interior point set
$\dot{\Gamma}_0 =\{t\beta_0+(1-t)\beta_1: 0< t< 1\}$.

To start, we prove that analytic continuation  is possible   when the points in $\dot{\Gamma}_0$ are close enough to the boundary $\mathbb{T}$. By the continuity of $u$ and the construction of $\Gamma$, we can choose
  a    number $r'$  close to $1$ such that $u(A_{r'})\subset A_r$ and $A_{r'}\cap \Gamma'=\emptyset$.  For each $0\leq k\leq n-1$, let
  $\widetilde{\rho}_k(z)=u^{-1} (\zeta^k u(z))$ when $z\in A_{r'} \,(\subseteq  u^{-1}(A_r))$. Fix  a point $z_0\in A_{r'} \cap [\mathbb{D}\backslash \Gamma]$,
  and let $U$ be a  small open disk  containing $z_0.$ Notice that both   $\{\rho_0,\rho_1,\cdots,\rho_{n-1}\}$ and
  $\{\widetilde{\rho}_0,\widetilde{\rho}_1,\cdots,\widetilde{\rho}_{n-1}\}$  are local inverses of $\phi^{-1} \circ \phi$ on $U$.
 So, after renumbering the local inverses if necessary,   we can suppose that $\rho_i=\widetilde{\rho}_i$ on $U$.
Since the domain $A_{r'} \cap [\mathbb{D}\backslash \Gamma]=A_{r'}\backslash \Gamma_0$ is connected and includes  $U$, one sees that $\rho_i=\widetilde{\rho}_i$ on this domain. Therefore,   the family of analytic functions $ \{\rho_i\cup \widetilde{\rho}_i\}$ defined as
\begin{displaymath}
 [\rho_i\cup \widetilde{\rho}_i](x)  = \left\{ \begin{array}{ll}
\rho_i(x) & \textrm{if $x\in \mathbb{D}\backslash \Gamma$}\\
\widetilde{\rho}_i(x) & \textrm{if $x\in A_{r'}$}
\end{array} \right.
\end{displaymath}
   are   local inverses  on   $A_{r'}\cup [\mathbb{D}\backslash \Gamma']$. We still denote them  by $\{\rho_i\}_i$ whenever no confusion arises.

Now let $S$ be a maximal subset of $\dot{\Gamma}_0$ on which these local inverses can't be  analytically continued   across. That is, $\{\rho_i\}_i$ are holomorphic  on the domain
$\mathbb{D}\backslash (\Gamma'\cup S) $, and can't be analytically continued  across    each point in $S$. We prove $S$ is empty by deriving contradiction. Indeed, assume $S$ is nonempty and let $$s= \inf \{t: t\beta_0+(1-t)\beta_1\in S\}.$$ Then $S$ is contained in the line segment from $z_0 =s\beta_0+(1-s)\beta_1$ to $\beta_1$.
Since $S\cap A_{r'}=\emptyset$, one sees that $0< s$ and $z_0$ is inside $\mathbb{D}$.  This means that one  can    analytically  extend the local inverses across $\{ t\beta_0+(1-t)\beta_1: t
 <s\} $, and  the process stops at $z_0.$ But, since $z_0$ is a regular point of $\phi$, there exists an   open disk $V=\{z:|z-z_0|<r_0\}$ with a small $r_0$,  such that
    $V \cap   \Gamma'=\emptyset$ and $\phi^{-1} \circ \phi$ has $n$ analytic branches on $V$.
 Notice that $$V \cap  [\mathbb{D}\backslash (\Gamma'\cup S)]=V\backslash S\supseteq V \backslash L,$$
where $L$ is a line segment  from the center $z_0$ to the boundary of the disk  $V$. It follows that $V \cap  [\mathbb{D}\backslash (\Gamma'\cup S)]$ is a connected domain. An  argument similar to that in the preceding paragraph shows that the local inverses are holomorphic  on $V \cup  [\mathbb{D}\backslash (\Gamma'\cup S)].$ By the maximality of $S$, we have that $V\cap S=\emptyset$, which leads to a  contradiction since $z_0\in \overline{S}$. Therefore, $S$ is empty and  the local inverses are holomorphic  on   $\mathbb{D}\backslash \Gamma'$, completing the proof.
\end{proof}

From the proof of the above lemma one derives an intrinsic order for the local inverses. Specifically, we label the local inverses $\{\rho_k(z)\}_{k=0}^{ n-1}$ such that  $\rho_k(z)=u^{-1} (\zeta^ k u(z)) $ on $\Omega$ for $0\leq k\leq n-1$. By a routine argument, we have that each $\rho_k$ is    invertible on $\Omega$, and for any pair  $\rho_k, \rho_{k'}$ and $ z\in \Omega$, we have
  $$\rho_k\circ \rho_{k'} (z)=\rho_{k+k' mod \,n} (z). $$
  Moreover, with little extra effort, one sees that each $\rho_k$ can also be analytically  continued   across the boundary $\mathbb{T}$.  We are now prepared to prove the main result.
\begin{thm}
Let $\phi$ be a finite Blaschke product of order $n$. Then the von Neumann algebra $\mathcal{A}_\phi=\{M_\phi,M_\phi^\ast\}'$ is commutative of dimension $q$, and hence $\mathcal{A}_\phi\cong\underbrace{\mathbb{C}\oplus\cdots\oplus \mathbb{C}}_q,$ where $q$ is the number   of connected components of the Riemann surface of $\phi^{-1}\circ \phi$.
\end{thm}
\begin{proof}
It suffices to show that   $ \mathcal{E}_j \mathcal{E}_i=\mathcal{E}_i\mathcal{E}_j$  for each $1\leq i,j\leq q$. Indeed,  for any $0\leq k,k'\leq n-1$, we have that $$\rho_k\circ \rho_{k'} (z)=\rho_k\circ \rho_{k'} (z)=\rho_{k+k' mod \,n} (z), z\in \Omega.$$
 Therefore,  for  any  $f\in L_a^2(\mathbb{D})$ and $z\in \Omega$, we have
\begin{eqnarray*}(\mathcal{E}_i \mathcal{E}_j f) (z)&=&\sum_{\rho\in G_i}\sum_{\widetilde{\rho}\in G_j} f(\widetilde{\rho}(\rho(z))) \widetilde{\rho}\,'(\rho(z))\rho\,'(z)
\\
&=&\sum_{\widetilde{\rho}\in G_j}\sum_{\rho\in G_i} f(\rho(\widetilde{\rho}(z))) \rho\,'(\widetilde{\rho}(z))\widetilde{\rho}\,'(z)=
(\mathcal{E}_j \mathcal{E}_i f) (z).\end{eqnarray*}
 This implies that $ \mathcal{E}_j \mathcal{E}_i (f)=\mathcal{E}_i\mathcal{E}_j(f)$ for any $f\in L_a^2(\mathbb{D})$,   completing the proof.
\end{proof}
By the final argument in the proof of \cite[Theorem 8.5]{DSZ}, the statement that $\mathcal{A}_\phi$ is commutative is equivalent to the statement that the minimal reducing subspaces for $\mathcal{M}_\phi$ are pairwise orthogonal. This also means that the number of  distinct minimal reducing subspaces of $M_\phi$ is equal to the dimension of $\mathcal{A}_\phi$.     Hence, one derives the following corollary giving the structure of the  reducing subspaces.
\begin{cor}
Let $\phi$ be a finite Blaschke product. Then the multiplication operator $M_\phi$ on the Bergman space $L_a^2(\mathbb{D})$  has exactly $q$ nontrivial minimal  reducing subspaces $\{\mathcal{M}_1,\cdots,\mathcal{M}_q\}$, and   $L_a^2(\mathbb{D})=\oplus_{k=1}^{q} \mathcal{M}_k$, where $q$ is the number of connected components of the Riemann surface $\phi^{-1}\circ \phi$.
\end{cor}

\section{Reducing subspaces}
In order to facilitate the comprehension of the rather involved computations included in the present section, we analyze first a simple, transparent example. If $\phi=z^n$, then the family of local inverses is $ \{\rho_k(z)=\zeta^k z: 0\leq k\leq n-1 \},$ and  we can infer without difficulty   that
$$\mathcal{M}_j=\overline{span}\{z^i:\,i\geq 0,\, i\equiv j \,(mod\, n)\} , 1\leq j\leq n  $$
are the minimal reducing subspaces of $M_{z^n} $. However, such a simple argument is not available in the general case,  so we prefer to explain the above  description of the $\mathcal{M}_j$  in a less direct way, as follows. Recall for $\phi=z^n$, we have that
$$(\mathcal{E}_k f)(z)=f( \rho_k(z))\,\rho_k'(z)=k \zeta^{k }   f(\zeta^k   \, z),\,1\leq k\leq n.$$
One verifies then that $\mathcal{M}_j$ is the joint eigenspace for the $\mathcal{E}_k\,'s$ corresponding to the eigenvalues $\zeta^ {k\,j }.$ Therefore, every $\mathcal{M}_j$ is a reducing subspace since the $\{\mathcal{E}_k\}$ are normal operators  and
$ \mathcal{A}_\phi= {\text{span}} \{\mathcal{E}_1,\cdots,\mathcal{E}_n\}.$

There is a second, more geometric description of $\mathcal{M}_j$ which  emerges from this simple example. Let $F_j$ be the flat bundle on $\mathbb{D}_0=\mathbb{D}\backslash \{0\}$ with respect to the jump $\zeta^j$
(see \cite{AD} for the precise definition). Roughly speaking, we cut   $\mathbb{D}_0$ along the line $(0,1)$ in $\mathbb{D}_0$, put the rank-one trivial holomorphic bundle over it, and identify the vector $v$ on the  lower copy of  $(0,1)$ with the vector $\zeta^j v$ on the above copy of  $(0,1)$. Then $F_j$ is just the quotient space obtained from this process. One can easily see that the  $F_j\,'s$  are all the flat line bundles whose pullback bundle  to $\mathbb{D}_0$ induced by the map $z^n: \mathbb{D}_0\to \mathbb{D}_0$ is the trivial bundle. This means that each holomorphic section on $F_j$ yields a holomorphic function on $\mathbb{D}_0$ by the induced composition. Let
$$L_a^2(F_j)=\{\text {holomorphic }s:\mathbb{D}_0 \to F_j :\int_{\mathbb{D}_0} |s|^2 dm<\infty\},$$
  and let $M_z$ be the corresponding bundle shift on $L_a^2(F_j)$. Note that $|s| $ is well defined on $\mathbb{D}_0$.
Then the operator $U_j: L_a^2(F_j)\to \mathcal{M}_j [\subseteq L_a^2(\mathbb{D})]$ defined by $(U_j f)(z)=n z^{n-1} f(z^n)$ is a unitary map, which intertwines $(L_a^2(F_j),M_{z})$ and $(\mathcal{M}_j,M_{z^n}).$ In this way flat line bundles provide  a natural model for the action of $M_{z^n}$ on the minimal  reducing subspaces of $M_{z^n}$.  It is conceivable that   some analogous  geometric description exists for the action of $M_\phi$ on the minimal reducing subspaces  in general, but, if so, we do not know how to describe it. Thus we follow a different  path below.

  Returning to the general case of a finite Blaschke product $\phi$, we will establish  the following main theorem in this section. Recall that the dual partition for $\phi$ is the partition of the set $\{0,1,\cdots,n-1\}$ for the equivalence relation defined  in $(\ref{dual}).$ We will prove lately
  that the number of  components in  the dual partition is also equal to $q$, the number of connected components of the Riemann surface for $\phi^{-1}\circ \phi$.
   \begin{thm}
 Let $\phi$ be a finite Blaschke product, and $\{G'_1,\cdots,G'_q\}$ be the dual partition for $\phi$. Then the multiplication operator $M_\phi$   has exactly $q$ nontrivial minimal  reducing subspaces $\{\mathcal{M}_1,\cdots,\mathcal{M}_q\},$ and for any $1\leq j\leq q$
 $$ \mathcal{M}_j=\{f\in {\mathcal O}(\mathbb{D}): f|_\Omega \in  \mathcal{L}_j^\Omega\},$$
 where $\Omega=u^{-1}(A_r)$ is defined in Lemma 2.1, and $\mathcal{L}^\Omega_j$ is a subspace of $L^2(\Omega) $ with the orthogonal basis  $ \{u^i u':   i+1(mod\,n) \in G'_j \}.$
 \end{thm}

The remainder  of this section is devoted to the proof of this theorem. We begin with  a characterization of the $\mathcal{M}_j\,'s$ in term of  eigenvalues and eigenspaces   of the $\mathcal{E}_k\,'s$.   Adapting, step by step, the proof of \cite[Theorem 8.5]{DSZ}, we infer that
 $$\mathcal{A}_\phi=\{M_\phi,M_\phi^\ast\}'= {\text{span}} \{\mathcal{E}_1,\cdots,\mathcal{E}_q\}= {\text{span}} \{P_{\mathcal{M}_1},\cdots,P_{\mathcal{M}_q}\},$$
 where   $P_{\mathcal{M}_k}$ is the projection onto $ \mathcal{M}_k$ for $1\leq k\leq q$.
 This means that there are unique constants $\{c_{ kj},1\leq  j,k\leq q\} $ such that \begin{equation}\label{ckj}\mathcal{E}_k=\sum_{1\leq  j\leq q} c_{ kj} P_{\mathcal{M}_j}.\end{equation} On the other hand, by a   dimension argument, the constant matrix $[c_{ kj}]$ is seen to be invertible. Since the rows of $[c_{kj}]$ are linearly independent, it follows that $c_{k\,j_1}=c_{k\,j_2}$ for each $k$ if and only if $j_1=j_2$.

 For each   tuple $\{c_{kj}\}_k$, let $\widetilde{\mathcal{M}}_j=\{f\in L_a^2(\mathbb{D}): \mathcal{E}_k f=c_{kj } f, \, 1\leq k \leq  q\}$ be the corresponding common eigenspace for $\{\mathcal{E}_1,\cdots,\mathcal{E}_q\}$. As shown in Theorem 2.3, each $\mathcal{E}_k$ is a normal operator. By  spectral theory, $\widetilde{\mathcal{M}}_{j_1} \bot \widetilde{\mathcal{M}}_{j_2}$ if $j_1\neq j_2$. By the fact that $ {\mathcal{M}_j}\subseteq  \widetilde{\mathcal{M}}_j$ for each $j$, we have that $\widetilde{\mathcal{M}}_{j } \bot  {\mathcal{M}}_{k}$ for $j \neq k$. Noticing that $L_a^2(\mathbb{D})=\oplus_k \mathcal{M}_k$, one sees that $ {\mathcal{M}_j}= \widetilde{\mathcal{M}}_j.$ That is,
\begin{equation}\label{sub} {\mathcal{M}_j}=\{f\in L_a^2(\mathbb{D}): \mathcal{E}_k f=c_{kj } f, \, 1\leq k \leq  q\}.\end{equation}

We also need the following lemmas concerning the  domain $\Omega=u^{-1}(A_r)$. Let
 $L_{a }^2(\Omega)$ be the Bergman space which consists of
the holomorphic  functions in  $L^2(\Omega)$, and let $L_{a,p}^2(\Omega)$ be the subspace of $L^2(\Omega)$ which is the closure of the polynomial ring in   $L^2(\Omega)$. Note that since $z^{-1}\in L^2(\Omega)$, we have that $L_{a,p}^2(\Omega)\neq L^2(\Omega).$ Recall that ${\mathcal O}(\mathbb{D})$ denotes the space of holomorphic functions on $\mathbb{D}$.
 \begin{lem} The restriction operator $i_{\Omega}: L_{a }^2(\mathbb{D})\to L_{a,p}^2(\Omega)$ defined by $i_{\Omega}(f)=f|_{\Omega}$
 is invertible. Furthermore, we have that  $L_{a }^2(\mathbb{D})=\{f\in {\mathcal O}(\mathbb{D}) : f|_\Omega\in L_{a}^2(\Omega)\}.$
 \end{lem}
 \begin{proof} As shown in the proof of Lemma $\ref{loc}$, there exists $r'>0$ such that $A_{r'}\subseteq \Omega$.
 It's well known that   there exists a positive constant $C_{r'}$ such that for any polynomial $f$
 $$\|f\|_ {L_{a }^2(\mathbb{D})}\leq C_{r'}\|f\|_ {L^2(A_{r'})}  .$$
 This implies for any polynomial $f$ that   $$  \|f\|_ {L ^2(\mathbb{D})}\leq C_{r'}\|f\|_ {L^2(A_{r'})} \leq C_{r'}\|f\|_ {L^2(\Omega)}  \leq C_{r'}\|f\|_ {L^2(\mathbb{D})}.$$
  Noticing that the polynomial ring is dense in both of the two Hilbert spaces $L^2_a(\mathbb{D})$ and $L_{a,p}^2(\Omega)$,  one sees that $i_{\Omega}$ is invertible.
%

In addition, we have that
$$L_{a }^2(\mathbb{D})=\{f\in {\mathcal O}(\mathbb{D}) : f|_\Omega\in L_{a,p}^2(\Omega)\}\subseteq \{f\in {\mathcal O}(\mathbb{D}) : f|_\Omega\in L_{a}^2(\Omega)\} .$$
It remains to show that, if  $f\in {\mathcal O}(\mathbb{D})$ and   $f|_\Omega\in L_{a }^2(\Omega)$, then $f\in L_{a }^2(\mathbb{D})$. Indeed, since $A_{r'}\subseteq \Omega$, one sees that $f|_ {A_{r'} }\in L_a^2(A_{r'})$. Let $f=\sum_{k=0}^\infty a_k z^k$ be the Taylor series expansion  for $f$ on $\mathbb{D}$. Since $\{z^k\}_k$ are pairwise orthogonal in $L_a^2(A_{r'})$, we have that the polynomial  $p_n=\sum_{k=0}^n a_k z^k$ tends to $f$ in the norm of $L_a^2(A_{r'})$ and hence $f\in L_{a,p}^2(A_{r'})$.
 Therefore, by the argument in the preceding paragraph,  there exists
$g\in L_{a }^2(\mathbb{D})$ such that $f|_{A_{r'}} =g|_{A_{r'}} $. This means that $f=g \in L_{a }^2(\mathbb{D})$, as desired.
\end{proof}
 Now we introduce  operators on $L_{a }^2(\Omega)$ and $L_{a,p}^2(\Omega)$ corresponding to $\{\mathcal{E}_i\}$. To simplify notation, we also let $M_\phi$ denote the multiplication operator on $L_{a }^2(\Omega)$ or $L_{a,p}^2(\Omega)$  with the bounded analytic symbol $\phi$. Recall  that each   $\rho \in \{\rho_j\}_{j=0}^{n-1} $ is invertible on  $\Omega$. Hence,   the operator $U^\Omega_{\rho }: L_a^2(\Omega) \to L_a^2(\Omega)$ defined by $U^\Omega_{\rho }(f)=(f\circ\rho )\,\rho\,' $ is a unitary operator with the inverse $U^\Omega_{\rho^{-1}}$. Similarly, for each $1\leq k\leq q$, define a linear operator
  $\mathcal{E}^\Omega_k:L_{a }^2(\Omega) \to L_{a }^2(\Omega)$ as
  $$\mathcal{E}^\Omega_k(f)=\sum_{ \rho \in G_k} U^\Omega_{\rho  }(f)=\sum_{ \rho \in G_k} (f\circ\rho  )\,\rho ',\,f\in L_{a }^2(\Omega).$$
Moreover, for each $f\in L_{a,p }^2(\Omega)$, there exists some $g\in L_{a }^2(\mathbb{D})$ such that $g|_\Omega=f$. A direct computation shows that
  $\mathcal{E}_k(g)|_\Omega=\mathcal{E}^\Omega _k(f)$. Hence, one sees that $\mathcal{E}^\Omega _k(f)\in L_{a,p }^2(\Omega) $. This means that $\mathcal{E}^\Omega_k$ is also a bounded  operator on  $L_{a,p }^2(\Omega) $ and $i_{\Omega} \mathcal{E}_k= \mathcal{E}^\Omega_k i_{\Omega}$. Combining this identity with  formula (\ref{ckj}) we obtain
  \begin{equation}\label{ockj}
  \mathcal{E}^\Omega_k (f)  =\sum_{1\leq  j\leq q} c_{ kj} i_{\Omega} P_{\mathcal{M}_j} i_{\Omega}^{-1} (f), \,f\in L_{a,p }^2(\Omega) .
  \end{equation}
   Furthermore, by \cite[Lemma 7.4]{DSZ}, for each $1\leq k\leq q$  there is an integer $k^-$ with $1\leq k^-\leq q$ such that
   $$G_{k^-}=G^-_k=\{\rho^{-1}:\rho \in G_k\} .$$
   Using an argument similar to that for \cite[Lemma 7.5]{DSZ}, we find that $\mathcal{E}^\Omega_{k^-}=\mathcal{E}^{\Omega\ast}_{k}  .$
   Therefore, $L_{a,p }^2(\Omega)$ is a common reducing subspace of $\{\mathcal{E}_k^\Omega\}$ and each $\mathcal{E}_{k}^\Omega$ is a normal operator on $  L_{a,p }^2(\Omega)$.

   For every $1\leq j\leq q$, let $$\mathcal{M}^\Omega_j=i_{\Omega}(\mathcal{M}_j)=\{f|_{\Omega}: f\in \mathcal{M}_j\} .$$  We claim   that
    $i_{\Omega} P_{\mathcal{M}_j} i_{\Omega}^{-1}= P_{\mathcal{M}^\Omega_j} $. Since the range of $ i_{\Omega} P_{\mathcal{M}_j} i_{\Omega}^{-1}$ is
    equal to $\mathcal{M}^\Omega_j$, it suffices to show that $i_{\Omega} P_{\mathcal{M}_j} i_{\Omega}^{-1}$  is a projection. Indeed, a direct computation
     shows that $i_{\Omega} P_{\mathcal{M}_j} i_{\Omega}^{-1}$ is an idempotent. Furthermore, combining formula (\ref{ockj}) and the fact  that $[c_{kj}]$ is invertible,
     every $i_{\Omega} P_{\mathcal{M}_j} i_{\Omega}^{-1}$ is a linear combination of $\{\mathcal{E}_{k}^\Omega\}.$ It follows that every
     $i_{\Omega} P_{\mathcal{M}_j} i_{\Omega}^{-1}$ is a normal operator. Therefore,  $i_{\Omega} P_{\mathcal{M}_j} i_{\Omega}^{-1}$ is a projection and $i_{\Omega} P_{\mathcal{M}_j} i_{\Omega}^{-1}= P_{\mathcal{M}^\Omega_j} $.

We summarize the consequences of the above argument as follows.
\begin{prop}
Using the notation above, $L^2_{a,p}(\Omega)=\oplus_{j=1}^q\mathcal{M}^\Omega_j,$ and
 \begin{equation}\label{ors} {\mathcal{M}^\Omega_j}=\{f\in L_{a,p}^2(\mathbb{D}): \mathcal{E}^\Omega_k f=c_{kj } f, \, 1\leq k \leq  q\}.\end{equation}
In addition, one has      \begin{equation}\label{ockj2}
  \mathcal{E}^\Omega_k (f)  =\sum_{1\leq  j\leq q} c_{ kj}   P_{\mathcal{M}_j}^{\Omega}  (f), \,f\in L_{a,p }^2(\Omega) .
  \end{equation}
  \end{prop}
  \begin{proof} Equation (\ref{ockj2}) follows from formula (\ref{ockj}) and the fact that $i_{\Omega} P_{\mathcal{M}_j} i_{\Omega}^{-1}= P_{\mathcal{M}^\Omega_j} $.  Combining this with the same argument in the beginning of the section,  one sees (\ref{ors}).

  Moreover, since $$P_{\mathcal{M}^\Omega_{i}}\,P_{\mathcal{M}^\Omega_{j}}=i_{\Omega} P_{\mathcal{M}_i} P_{\mathcal{M}_j} i_{\Omega}^{-1} =0$$ if $i\neq j$ and
  $$\sum_{j=1}^q  P_{\mathcal{M}^\Omega_{j }}=\sum_{j=1}^q  i_{\Omega} P_{\mathcal{M}_j} i_{\Omega}^{-1}=I,$$ we have that $L^2_{a,p}(\Omega)=\oplus_j\mathcal{M}^\Omega_j,$ completing the proof.
  \end{proof}
   Since   $\rho_1$ is invertible and $\rho_1^n=1$ on $\Omega$, the operator $U^\Omega_{\rho_1 }: L_a^2(\Omega) \to L_a^2(\Omega)$ is unitary and $(U^\Omega_{\rho_1 })^n=1$.    By the spectral theory for unitary operators, the $\{\zeta^i\}_{i=0}^{n-1}$ are possible  eigenvalues of $U^\Omega_{\rho_1 }$, and
   $ U^\Omega_{\rho_1 }  =\sum_{i=0}^{n-1} \zeta^{i } P_{\mathcal{N}^\Omega_i},$
   where $ P_{\mathcal{N}^\Omega_i}$ is the projection from $L_{a }^2(\Omega)$ onto the eigenvector subspace
   $$\mathcal{N}^\Omega_i=\{f\in L_a^2(\Omega): U^\Omega_{\rho_1 }(f)=\zeta^i f\}.$$
    It follows that $U^\Omega_{\rho_j }=(U^\Omega_{\rho_1 } )^j=\sum_{i=0}^{n-1} \zeta^{i\,j} P_{\mathcal{N}^\Omega_i},$
    and \begin{equation}\label{ockj3}\mathcal{E}^\Omega_k(f) =\sum_{\rho_j\in G_k} \sum_{i=0}^{n-1} \zeta^{i\,j} P_{\mathcal{N}^\Omega_i} (f),\, f\in L_{a }^2(\Omega). \end{equation}
    Furthermore,  we have the following lemma. Recall that $u:\Omega=u^{-1}(A_r)\to A_r$ is invertible    as shown  in Lemma 2.1.
\begin{lem}
$\mathcal{N}^\Omega_i=\overline{span} \{u^k u': k\in\mathbb{Z}, k+1\equiv i\, mod\, n  \}$.
\end{lem}
\begin{proof}
Since $u\circ \rho_1=\zeta u$ on $\Omega$, it is easy to check that $$U_{\rho_1} (u^k u')=\zeta^i u^k u',  \text{ for } k+1\equiv i\, mod\, n.$$ That is, $\mathcal{N}_i^\Omega$ is contained in the eigenspace of $U_{\rho_1}$ for the eigenvalue  $\zeta^i$. It remains to show that $\oplus_i \mathcal{N}^\Omega_i=L_a^2(\Omega)$. In fact, we will prove that
  $\{u^k u': k\in\mathbb{Z}\} $ is a complete  orthogonal basis for $L_a^2(\Omega)$.

Define the pull-back operator $C_u:   L_a^2(A_r)\to L_a^2(\Omega) $ by
$$C_u f=(f\circ u) \,\,u'.$$
Since $u: \Omega\to A_r$ is invertible, $C_u$ is unitary. Noticing that $\{z^k:k\in\mathbb{Z}\}$ is a complete  orthogonal basis for $L_a^2(A_r)$, one sees that
$\{u^k u'=C_u(z^k): k\in\mathbb{Z}\} $ is a complete  orthogonal basis for $L_a^2(\Omega)$, as desired.
\end{proof}

 Recall  that  for   the partition $\{G_1,\cdots,G_q\}$      of local inverses for $\phi^{-1}\circ \phi$,     we say   $j_1\sim j_2 $ in the dual partition for two integers $0\leq j_1,j_2\leq n-1,$ if
  $$\sum_{\rho_k\in G_i}\zeta^{k\,j_1\,}=\sum_{\rho_k\in G_i}\zeta^{k\,j_2 }  \text{ for any  } 1\leq i\leq q.$$
By this  equivalence relation,   the set $\{0,1,\cdots,n-1\}$ is partitioned into equivalence classes
  $\{G'_1,\cdots,G'_p\}.$

 For each  $G'_j$ in the dual partition, let  $\mathcal{L}^\Omega_j=\oplus_{i\in G'_j} \mathcal{N}^\Omega_i $; that is,
 $$\mathcal{L}^\Omega_j=\overline{span} \{u^i u': i\in\mathbb{Z}, i+1(mod\,n) \in G'_j \}. $$
Then $\oplus_{j=1}^p \mathcal{L}^\Omega_j=L_a^2(\Omega)$. From formula (\ref{ockj3})
 \begin{equation}
 \label{E2} \mathcal{E}^\Omega_k(f) =\sum_{1\leq j\leq p} c'_{k\,j}  P_{\mathcal{L}^\Omega_j} (f),\, f\in L_{a}^2(\Omega),
 \end{equation}
 where $c'_{k\,j}=\sum_{\rho_i\in G_k} \zeta^{i\,l}$ for any $l\in G'_j$. By  the equivalent condition for the dual partition,   $c'_{k\,j_1}=c'_{k\,j_2}$ for each $k$ if and only if $ {j_1}= {j_2} $. Comparing formulas (\ref{ors}) and (\ref{E2})   yields the following result.
 \begin{prop}
 For each $\mathcal{M}^\Omega_j$, there exists $  1\leq k\leq p$ such that $\mathcal{M}^\Omega_j=\mathcal{L}^\Omega_{ k} \cap L_{a,p}^2(\Omega).$ 
 \end{prop}
 \begin{proof}
  For each $0\neq f\in \mathcal{M}^\Omega_j\subseteq\oplus_k \mathcal{L}_k^\Omega=L_a^2(\Omega), $  there exists at least one $d_f$ such that $1\leq d_f\leq p$  and the projection of $f$ on  $ \mathcal{L}^\Omega_{d_f} $ is nonzero. We claim that $d_f$ is unique.
   Indeed, suppose  for    $  k_1\neq k_2 $,  $P_{\mathcal{L}^\Omega_{k_1}}(f)$ and $ P_{\mathcal{L}^\Omega_{k_2}}(f) $ are nonzero.
 By formula (\ref{ors}), one sees  for each $1\leq i\leq n$ that,
 $$[P_{\mathcal{L}_{k_1}}+P_{\mathcal{L}_{k_2}}]\mathcal{E}^\Omega_i(f)=c_{ ij} P_{\mathcal{L}_{k_1}} (f)+c_{ ij} P_{\mathcal{L}_{k_2}} (f).$$
  Moreover, by formula (\ref{E2}),
  $$[P_{\mathcal{L}_{k_1}}+P_{\mathcal{L}_{k_2}}]\mathcal{E}^\Omega_i(f)=c'_{i\,k_1} P_{\mathcal{L}_{k_1}} (f)+ c'_{i\,k_2} P_{\mathcal{L}_{k_2}} (f).$$
  This implies that $c_{ ij}=c'_{i\,k_1 }=c'_{i\,k_2 }$ for each $i$. This leads to an contradiction since $k_1\neq k_2$. Therefore,  there exists only one integer $d_f$ such that $P_{\mathcal{L}_{d_f}^\Omega}(f)\neq   0 .$

 We now prove that $d_f$ is independent of $f$. Otherwise, there exist  $k_1\neq k_2$ and $f_1,f_2\in \mathcal{M}_j$ such that both $P_{\mathcal{L}^\Omega_{k_1}}(f_1)$ and $P_{\mathcal{L}^\Omega_{k_2}}(f_2) $ are nonzero. By the uniqueness proved in the preceding paragraph, we have that $P_{\mathcal{L}^\Omega_{k_1}}(f_2)= P_{\mathcal{L}^\Omega_{k_2}}(f_1)=0. $ However,  this means that
 both  $ P_{\mathcal{L}^\Omega_{k_1}}(f_1+f_2)$ and $ P_{\mathcal{L}^\Omega_{k_2}}(f_2+f_1)$ are nonzero, which contradicts   the uniqueness of $d_{f_1+f_2}$.

   Therefore, there exists only one integer $k$ such that  $P_{\mathcal{L}^\Omega_k} \mathcal{M}^\Omega_j\neq \{0\}$. Moreover,  we have that $c_{ij}=c'_{ik}$ for each $i$.
        Combining this fact with formulas (\ref{ors}) and (\ref{E2}),
    one sees  that $$\mathcal{M}^\Omega_j=  \mathcal{L}^\Omega_k\cap L_{a,p}^2(\Omega)=\{f\in L_{a,p}^2(\mathbb{D}): \mathcal{E}^\Omega_i f=c_{ij } f, \, 1\leq i \leq  q\},$$ completing the proof.
 \end{proof}

 In what follows, we will prove the converse of the above proposition. We begin with some lemmas.

 \begin{lem}
 Let $f$ be a function  holomorphic    on a neighborhood  of  $\overline{A_r}$. Then for any $k\in\mathbb{Z}$, $f\bot z^k$ in $L^2_a(A_r)$ if and only if
 $ \int_{z\in\mathbb{T}} f(z) \overline{z^k} dm(z)=0. $
 \end{lem}
 \begin{proof}
 Let $a_k$ be the coefficient for $z^k$ in the Laurent series expansion of $f$ on $A_r$. Observe that $\{z^k\}^{+\infty}_{k=-\infty}$ is a complete orthogonal basis for both of $L_a^2(A_r)$ and $L^2(\mathbb{T})$.   A direct computation shows that  $\langle f, z^k\rangle_{L^2_a(A_r)}= a_k  \|z^k\|_{L^2_a(A_r)}$
  and
  $\langle f, z^k\rangle_{L^2{(\mathbb{T})}}= a_k  \|z^k\|_{L^2{(\mathbb{T})}},$ which leads to the desired result.
 \end{proof}
 We also need the following    transformation formula. 
 \begin{lem}
 Let $s:\mathbb{T}\to\mathbb{T}$ be an invertible differentiable map. Then there exists a constant $ \epsilon_s=1$ or $-1$, such that for any $f\in C(\mathbb{T}) $
 $$\int_ \mathbb{T}f(\theta) dm(\theta)= \epsilon_s\int_ \mathbb{T}  f( s(\theta)) \frac{  s'(\theta)}{i\,s(\theta)} dm(\theta).$$
 If, in addition, $s$ is   holomorphic  on a neighborhood  of $\mathbb{T}$, then
 $$\int_ \mathbb{T}f(z) dm(z)= \epsilon_s\int_ \mathbb{T}  f( s(z)) \frac{ z\, s'(z)}{s(z)} dm(z).$$
  \end{lem}
 \begin{proof}
It is sufficient to verify only the first  equation. Indeed,
 the latter equation follows from the former equation by the fact that $$s'(\theta)=s'(z) \frac{dz}{d\theta} =i\,e^{i\theta} s'(z)=i\,z\,s'(z),\,z\in \mathbb{T}.$$
 Without loss of generality, we can suppose  that $s(1)=1$. Then there exists   $\widetilde{s}:(0,2\pi) \to (0,2\pi)$ such that $s(\theta)=e^{i \widetilde{s}(\theta)}$. An elementary calculus argument  shows that
 $$\int_ \mathbb{T}f(\theta) dm(\theta)=  \int_ \mathbb{T}  f( s(\theta)) | \widetilde{s}\,' (\theta) | dm(\theta) .$$
 Since $s$ is invertible on $\mathbb{T}$, one has that $\widetilde{s}:(0,2\pi) \to (0,2\pi)$  is a monotonic function. Therefore, we can choose a constant $ \epsilon_s=1$ or $-1$  such that $|\widetilde{s}\,'|=\epsilon_s \widetilde{s}\,' $. Moreover, differentiating the equation $s(\theta)=e^{i \widetilde{s}(\theta)}$, one sees that $ s'(\theta)=i\,e^{i \widetilde{s}(\theta)} \,\widetilde{s}\,'(\theta)=i\,s(\theta) \,\widetilde{s}\,'(\theta)$. This implies  that $| \widetilde{s}\,' (\theta) |=\frac{  \epsilon_s s'(\theta)}{i\,s(\theta)}$, completing the proof.
 \end{proof}
 \begin{lem}
 For  any  integer $ k \geq 0$,  there exists some integer $i\geq 0$ such that $\langle z^i, u^k u'\rangle_{L^2(\Omega)}\neq 0$.  Therefore, $P_{ {L}_{a,p}^2(\Omega)} \mathcal{N}^\Omega_k\neq \{0\}$ for all  $0\leq k\leq n-1$.
 \end{lem}
 \begin{proof}
 We prove the statement by contradiction. Suppose that  for some $k\geq 0$,   $$ \langle z^i, u^k u'  \rangle_{L^2(\Omega)} =0,\,\, \forall i\geq 0 .$$
Since the operator $C_u:L^2(A_r) \to L^2(\Omega)$, which  appears  in Lemma 3.4, is   unitary, the above equation is equivalent to
$$ \langle (u^{-1})^i (u^{-1})', z^k  \rangle_{L^2(A_r)} =0,\,\, \forall i\geq 0 .$$
  Using Lemma 3.6, it follows that for each integer $i\geq 0$
$$ \langle (u^{-1})^i (u^{-1})', z^k  \rangle_{L^2(\mathbb{T})} =\int_{\mathbb{T}} (u^{-1})^i (u^{-1})' \overline{z^k} dm(z)=0 .$$
By Lemma 3.7, Lemma 2.1 and the fact that $|u(z)|=1$ for $z\in\mathbb{T}$, we have for each integer $i\geq 0$, that
$$0=\int_{\mathbb{T}} z^i \, (u^{-1})'\circ u(z) \,  \overline{u^k}  \frac{ z\, u'(z)}{u(z)} dm(z)=\int_{\mathbb{T}} z^{i+1} \overline{u^{k+1}}dm(z)=\langle z^{i+1}, u^{k+1}   \rangle_{L^2(\mathbb{T})}.  $$
 This means that   $u^{k+1} \in \overline{H_2(\mathbb{T}})$ and  hence $\phi^{k+1}=u^{n(k+1)} \in \overline{H_2(\mathbb{T}})$.  Noticing that
$\phi^{k+1}$ is   holomorphic  on $\mathbb{D}$, one sees that $\phi^{k+1}$ is a constant. This leads to a  contradiction since $\phi$ is a nontrivial Blaschke product, completing the proof.
 \end{proof}
 Summarizing   the above results, we obtain the converse of Proposition 3.5.
 \begin{prop}
For each $k$, there exists a unique $j$ such that $\mathcal{M}^\Omega_j=\mathcal{L}^\Omega_k \cap L_{a,p}^2(\Omega) $; that is,
 $$L_{a,p}^2(\Omega)=\oplus_{k} [ \mathcal{L}^\Omega_k \cap L_{a,p}^2(\Omega)]. $$
 \end{prop}
 \begin{proof}
  From Proposition 3.5, for each $1\leq j\leq q$, there exists only one $1\leq k_j\leq p$ such that $\mathcal{M}_j^\Omega=   \mathcal{L}^\Omega_{k_j} \cap L_{a,p}^2(\Omega) . $  Hence,
  $$L_{a,p}^2(\Omega)=\oplus_{j} [ \mathcal{L}^\Omega_{k_j} \cap L_{a,p}^2(\Omega) ] .$$
 We claim that the set $\{k_1,\cdots,k_q\}$ is just $\{1,\cdots,p\}$. Indeed, if there exists $k$ such  that $1\leq k\leq p$ but $k$ is not in the set  $\{k_1,\cdots,k_q\}$, 
then  $\mathcal{L}_k^\Omega \bot \oplus_{k_j} \mathcal{L}^\Omega_{k_j}$. This means that $P_{L_{a,p}^2(\Omega)} \mathcal{L}_j^\Omega= \{0\} $, which  leads to a  contradiction, since   $\mathcal{L}_k^\Omega=\oplus_{j\in G'_k} \mathcal{N}^\Omega_j$ and by Lemma 3.8 we have that $P_{ {L}_{a,p}^2(\Omega)} \mathcal{N}^\Omega_j\neq \{0\}$ for each $j$. Therefore, the set $\{k_1,\cdots,k_q\}$ includes all integers between $1$ and $p$. It follows that  $p=q$ and
  $$L_{a,p}^2(\Omega)=\oplus_{k=1}^{q} [ \mathcal{L}^\Omega_{ k} \cap L_{a,p}^2(\Omega) ] ,$$
  as desired.
 \end{proof}
In the proof of Proposition 3.9, one identifies    the following  intrinsic  property of the partition   for a finite Blaschke product.
 \begin{cor}
 The number of  components in  the dual partition is also equal to $q$, the number of connected components of the Riemann surface for $\phi^{-1}\circ \phi$.
 \end{cor}
Combining Lemma 3.2 with Propositions 3.5 and  3.9,  we derive our main result in this section.\\
{\bf{Proof of Theorem 3.1.}}
  Combining Propositions 3.5 and  3.9, after renumbering if necessary, we have  for each $1\leq j\leq q$ that,
   $$\mathcal{M}_j^\Omega=   \mathcal{L}^\Omega_{j } \cap L_{a,p}^2(\Omega) . $$
   Noting that $i_\Omega$ is invertible, one sees that
   $$\mathcal{M}_j=\{f\in L^2_a(\mathbb{D}):\,\, f|_\Omega\in \mathcal{M}_j^\Omega\} =\{f\in L^2_a(\mathbb{D}):\,\, f|_\Omega\in \mathcal{L}^\Omega_{j }\}.$$
   Combining this formula with Lemma 3.2, we have that
   $$\mathcal{M}_j= \{f\in {\mathcal O}(\mathbb{D}): f_\Omega\in \mathcal{L}^\Omega_{j }\},$$
   completing the proof of the theorem.
 $\quad\quad\quad\quad\quad\quad\quad\quad\quad\quad\quad \quad\quad\quad\quad\quad\quad\quad\quad\quad\quad\Box$

\section{Arithmetics of reducing subspaces}
 In \cite{GSZZ,SZZ}, the authors obtained a classification of the structure of the finite Blaschke product $\phi$  in   case $\phi$ has order $3$ or $4$.   In this section we show an arithmetic way towards the classification
of    finite Blaschke products, displaying the details  for the case  of order $8$.

Following \cite{DSZ} we define an equivalence relation among   finite Blaschke products so that   $\phi_1\sim\phi_2$, if there exist
  M$\ddot{o}$bius transformations $\varphi_a(z)=\frac{a-z}{1-\overline{a}z}$ and $\varphi_b(z)=\frac{b-z}{1-\overline{b}z}$ with $a,b\in\mathbb{D}$ such that
$\phi_1=\varphi_a\circ \varphi_2 \circ \varphi_b.$
 A finite Blaschke $\phi$ is called {\it reducible} if there exist two nontrivial finite Blaschke products $\varphi_1,\varphi_2$ such that $\phi\sim \varphi_1\circ \varphi_2$,  and $\phi$ is  {\it irreducible} if $\phi$ is not reducible.

For a finite Blaschke product $\phi$ of order $n$, let $G_1,\cdots,G_q$ be the partition defined by  the family of local inverses $\{\rho_0,\cdots,\rho_n\}$ for $\phi^{-1}\circ \phi$. When no confusion arises,
we write   $i\in G_k$ if $\rho_i\in G_k$, and $G_k=\{i_1,i_2,\cdots,i_j\}$ if $G_k=\{\rho_{i_1},\rho_{i_2},\cdots,\rho_{i_j}\}$. In view of the above notations, $\{G_1,\cdots,G_q\}$ is a partition of the additive group $\mathbb{Z}_n=\{0,1,\cdots,n-1\}$. One can immediately verify that, if $\phi_1\sim\phi_2$, then $\phi_1, \phi_2$ yield identical partitions.

The result in Corollary 3.10 hints that there should exist some internal algebraic and combinatorial  structures for the partitions arising from   finite Blaschke products.  Although we don't understand these  properties  completely, we list a few necessary conditions:

{\emph{$(\alpha_0)$  $\{0\}$ is a singleton in the partition,  since $\rho_0(z)=z$ is holomorphic on $\mathbb{D}$. }}

{\emph{$(\alpha_1)$ For any pair $G_i$ and $G_j$, there exist some $G_{k_1},\cdots,G_{k_m}$ such that
   $$G_i+G_j= G_{k_1}\cup\cdots\cup G_{k_m}  \,\,(\text{counting \,\,multiplicities on both sides}), $$
   where $"+"$ is defined using  the addition of $\mathbb{Z}_n$. (This is a consequence of the fact that the product $\mathcal{E}_i \mathcal{E}_j$ is a linear combination of some $\mathcal{E}_k\,'s $).}}

{\emph{$(\alpha_2)$  By \cite[Lemma 7.4]{DSZ}, for each $G_i=\{i_1,\cdots,i_k\}$, there exists $j$ such that
    $$G_j=G_i^{-1}=\{n-i_1,\cdots,n-i_k\}. $$}}
{\,\,\,\emph{$(\alpha_3)$ By Corollary 3.10, the number of elements in the dual partition is also $q$. } }

 We also need the following generalization of \cite[Lemma 8.3]{DSZ}. Note that the additive structure for elements in $G_k\,'s$ coincides   with compositions near the boundary $\mathbb{T}$.
\begin{lem}
For a finite Blaschke product $\phi$ of  order $n$, $\phi$ is reducible if and only if $G_{k_1}\cup\cdots\cup G_{k_m}$ forms a nontrivial  proper subgroup of $\mathbb{Z}_n$, for some subset
$G_{k_1},\cdots,G_{k_m}$ of the partition arising from $\phi$.
\end{lem}
\begin{proof}
Assume  that $\phi$ is reducible.  Without loss of generality,   suppose that $\phi=\varphi_1\circ \varphi_2$ for two nontrivial  finite Blaschke products $\varphi_1,\varphi_2$.
Since the family of local inverses $\varphi_2^{-1}\circ \varphi_2$ is a cyclic group under  compositions near the boundary $\mathbb{T}$, and it is contained in the local inverses  of $\phi^{-1}\circ \phi$,  the set of  the local inverses for $\varphi_2^{-1}\circ \varphi_2$ forms a nontrivial proper subgroup of $\phi^{-1}\circ \phi$.

On the other hand, suppose that $G=G_{k_1}\cup\cdots\cup G_{k_m}$ is a nontrivial proper subgroup of $\mathbb{Z}_n$ for some
$G_{k_1},\cdots,G_{k_m}$. For each $G_{k_i}=\{\rho_{i_1},\cdots,\rho_{i_j}\},$ by \cite[Thereom 3.1]{DSZ} there exists a polynomial $f_i(w,z)$ of degree $ j$ such that
$\{ \rho_{i_1}(z) ,\cdots,\rho_{i_j}(z)\}$ are solutions of $f_i(w,z)=0$. This implies that $\prod_{\rho\in G_{k_i}} \rho(z)=\frac{p_i(z)}{q_i(z)}$ is a quotient of two polynomials $p_i(z),q_i(z)$ of degree at most  $ j$. So, if we define
$$\varphi_2(z)= \prod_{\rho\in G } \rho(z)=\prod_{i=1}^m\prod_{\rho\in G_{k_i}} \rho(z) =\prod_{i=1}^m \frac{p_i(z)}{q_i(z)},$$
then $\varphi_2(z)$ is a rational function of degree at most  $\sharp G$; here $\sharp G $ denotes the number of elements in $G$. It follows that $\varphi_2(z)$ is holomorphic outside a finite point set $S$ of $\mathbb{D}$. Since  each  local inverse  is bounded by $1$ on $\mathbb{D}\backslash \Gamma'$ and   $\mathbb{D}\backslash \Gamma'$ is dense in $\mathbb{D}$, we have that $\varphi_2$ is also bounded on $\mathbb{D}\backslash S$ and  hence it can be analytically  continued across $S$. This means that $\varphi_2$ is a bounded holomorphic   function on $\mathbb{D}$. From a similar  argument involving  local inverses, one sees that $\varphi_2$ is also continuous  on $\mathbb{T}$ and $|\varphi_2(z)|=1$ whenever $z\in\mathbb{T}$. That implies $\varphi_2$ is a finite Blaschke product of order $\sharp G $.

Furthermore, by the group structure of $G$,
  $\varphi_2(\rho_i(z))=\varphi_2(z) $ for each $\rho_i\in G$ if $z$ is close enough to the boundary $\mathbb{T}$. Since $\mathbb{D}\backslash \Gamma'$ is a connected domain including $\Omega$, the equation  still holds whenever  $z\in \mathbb{D}\backslash \Gamma'$.  In other words,    the family of local inverses of $\varphi_2^{-1}\circ \varphi_2$ is just, $G$, a subset in that of $\phi^{-1}\circ \phi$. Consequently, $\phi(z_1)=\phi(z_2)$ if $\varphi_2(z_1)=\varphi_2(z_2)$ and $z_1,z_2$ are
  regular points of $\varphi$. Hence, if  we   define
  $$\varphi_1(w)=\phi(z)  \,\,\text{for}  \, w=\varphi_2(z),$$
  then $\varphi_1$ is well defined outside some finite set of points in $\mathbb{D}$. With a similar argument for $\varphi_2$, one sees that $\varphi_1$ is also a finite Blaschke product, which satisfies $\phi=\varphi_1\circ \varphi_2$, completing the proof of the lemma.
\end{proof}
By the above proof, one sees that if $\phi$ is reducible, then some of the local inverses can be analytically continued across some critical points of $\phi$. But it is not clear that this is a sufficient condition for $\phi$ to be reducible.

Based on the above lemma, we explain the classification for a general  Blaschke product of order four.\\ \\
{\bf{[11, Theorem 2.1.]}} {\emph{Let $\phi$ be a Blaschke product  of  order $4$. One of the following scenarios holds.}}

{\emph{$\mathrm{(1)}$ The partition of $\phi$ is $ \{\{0\},\{1\},\{2\},\{3\}\} $; equivalently,  $\phi\sim z^4$.}}

{\emph{$\mathrm{(2)}$ The partition of $\phi$ is $ \{\{0\}, \{2\},\{1,3\} \}$; equivalently,  $\phi\sim \phi_a^2(z^2)$, where $\phi_a=\frac{a-z}{1-\overline{a}z}$ is a M$\ddot{o}$bius transformation with   $a\neq 0$. }}

{\emph{$\mathrm{(3)}$ The partition of $\phi$ is $\{ \{0\},  \{1,2,3\} \}$; equivalently,   $\phi$ is not  reducible.}}\\
\\
All above possibility occur for some $\phi$ by the result  of Sun, Zheng and Zhong in \cite{SZZ}.

We now classify, using purely arithmetical considerations, the possible structure for a finite Blaschke product of order eight.
\begin{thm}
 Let $\phi$ be a Blaschke product of order $8$. One of the following scenarios holds.

 $\mathrm{(1)}$ The partition of $\phi$ is $ \{\{0\},\{1\},\{2\},\{3\},\{4\},\{5\},\{6\},\{7\} \}$; equivalently,  $\phi\sim z^8$.

 $\mathrm{(2)}$ The partition of $\phi$ is $\{ \{0\}, \{2\},\{4\},\{6\},\{1,5\},\{3,7\} \}$; equivalently,  $\phi\sim \phi_a^2(z^4)$, where $\phi_a=\frac{a-z}{1-\overline{a}z}$ is a M$\ddot{o}$bius transformation with   $a\neq 0$.

  $\mathrm{(3)}$ The partition of $\phi$ is $\{ \{0\},  \{4\},\{1,2,3,5,6,7\}\}$; equivalently,   $\phi\sim \varphi(z^2 )$, where $\varphi$ is an irreducible Blaschke product of order $4$.

 $\mathrm{(4)}$ The partition of $\phi$ is one of $\{ \{0\},  \{4\},\{2,6\},\{1,3,5,7\}\}$, $\{ \{0\},  \{4\},\{2,6\},\{1,3\},\{5,7\}\}$, $\{ \{0\},  \{4\},\{2,6\},\{1,5\},\{3,7\}\}$ or  $\{ \{0\},  \{4\},\{2,6\},\{1,7\},\{3,5\}\}$; equivalently,   $\phi\sim \psi(  \varphi_a^2(z^2)  ) $, where $\psi$ is a  Blaschke product of order $2$ and  $\phi_a=\frac{a-z}{1-\overline{a}z}$ is a M$\ddot{o}$bius transformation with   $a\neq 0$.

 $\mathrm{(5)}$ The partition of $\phi$ is $\{ \{0\}, \{2,4,6\},\{1,3,5,7\}\}$; equivalently,   $\phi\sim  \psi\circ  \varphi  $, where $\psi$ is a  Blaschke product of order $2$ and $\varphi$ is an  irreducible Blaschke product of order $4$.

  $\mathrm{(6)}$ The partition of $\phi$ is $\{ \{0\}, \{1,2,3,4,5,6,7\}\}$; equivalently,   $\phi$ is not reducible.
\end{thm}
A similar approach would work for Blaschke products of arbitrary order.  However, it seems difficult to decide whether a partition satisfying conditions $(\alpha_0)$,   $(\alpha_1)$,
 $(\alpha_2)$ and  $(\alpha_3)$ arises from a finite Blaschke product. For example, we cannot exhibit examples for each partition in Case $(4)$ in Theorem 4.1, although it is likely that they exist. We now prove Theorem 4.1 in what follows.

{\bf{Proof of Theorem 4.1.}}
By condition $(\alpha_0)$, $\{0\}$ is a singleton in the partition for $\phi$. Without loss of generality, suppose that $G_1=\{0\}$. We list all possibilities by
the minimal number $s=\min\{\sharp G_2,\cdots,\sharp G_q\}, $ where $\sharp G_k$ is the number of elements in $G_k$. Clearly $s\neq 4,5,6.$
\vskip2mm {\bf (I) Case   $s=1.$}  We suppose without loss of generality that $G_2$   is also a singleton .

{\bf Subcase (A):} suppose $G_2$ consists of one of the primitive element   $\{1,3,5,7\}$ in $\mathbb{Z}_8$. Since $\mathbb{Z}_8$ is generated by any element in $\{1,3,5,7\}$,  by Conditions $(\alpha_1)$ and $(\alpha_2)$,  each $G_k$ is a singleton. That is, the partition is just $ \{\{0\},\{1\},\{2\},\{3\},\{4\},\{5\},\{6\},\{7\} \}$.
By \cite[Lemma 8.1 and Lemma 8.3]{DSZ}, one sees that this  is equivalent  to   $\phi\sim z^8$.
\vskip2mm
{\bf Subcase (B):} suppose it is not subcase (A) and $G_2$ consists of $2$ or $6$. By Condition $(\alpha_1)$, the partition contains the singletons $\{2\},\{4\},\{6\}$. We list all possible partitions as follows:
\begin{enumerate}
\item[(B1)]\,\,\quad $\{ \{0\}, \{2\},\{4\},\{6\},\{1,5,3,7\} \}$;
\item[(B2)]\,\,\quad $\{ \{0\}, \{2\},\{4\},\{6\},\{1,3\},\{5,7\} \}$;
\item[(B3)]\,\,\quad $\{ \{0\}, \{2\},\{4\},\{6\},\{1,5\},\{3,7\} \}$;
\item[(B4)]\,\,\quad $\{ \{0\}, \{2\},\{4\},\{6\},\{1,7\},\{3,5\} \}$.
\end{enumerate}
Case  $(B2)$ is excluded by Condition $(\alpha_1)$, since   $\{2\}+\{1,3\}=\{3,5\}$ is not a union of some $G_k$ in $(B2)$. One can get rid of $(B4)$ in a similar way. The remaining cases, $(B1)$ and  $(B3)$, satisfy $(\alpha_0)$, $(\alpha_1)$ and $(\alpha_2)$. But, by a direct computation they have the same dual partition
$\{ \{0\}, \{2\},\{4\},\{6\},\{1,5\},\{3,7\} \}$. Using Condition $(\alpha_3)$, we have that $\{ \{0\}, \{2\},\{4\},\{6\},\{1,5\},\{3,7\} \}$ is the unique choice. In this case, by Lemma 4.1, there exist a finite Blaschke product $\varphi_1$ of order $4$ and  a finite Blaschke product  $\varphi_2$ of  order $2$ such that $\phi=\varphi_2\circ \varphi_1$. Moreover, by the proof of Lemma 4.1,   local inverses for $\varphi_1$ are $\rho_0,\rho_2,\rho_4,\rho_6$  in the family of local inverses of $\phi$. By \cite[Lemma 8.1 and Lemma 8.3]{DSZ}, one sees that this condition  is equivalent  to    $\varphi\sim z^4$. This means that $\phi\sim \psi(z^4)$ for some Blaschke product $\psi$ of order $2$.   Observe  that two local inverses for $\psi$ are holomorphic on $\mathbb{D}$, since one of them, $\rho_0(z)=z$, is holomorphic.  By \cite[Lemma 8.1 and Lemma 8.3]{DSZ}, $\psi=\phi_b\circ z^2\circ \phi_a$ for some M$\ddot{o}$bius transforms $\phi_a,\phi_b$. This implies that $\phi\sim \phi_a^2(z^4)$, and $a\neq 0$, since it would   degenerate  to subcase $(A)$ if a=0.
\vskip2mm
We now consider the most complicated case in which  $G_2=\{4\}$ is the unique singleton other than $G_1$. We divide it into several different subcases looking again  at the minimal number $t=\min\{\sharp G_3,\cdots,\sharp G_q\}. $ Clearly $2\leq t\leq 5$ and $t\neq 4$. So, $t$ is $2$, $3$, or $5$.

{\bf Subcase (C):}   $G_1=\{0\},G_2=\{4\}$ and $t=5$.

The only possibility is the partition   $\{ \{0\},  \{4\},\{1,2,3,5,6,7\}\}$. By Lemma 4.1 and the observation that $\psi\sim z^2$ for each
 Blaschke product $\psi$ of  order $2$, one sees that there exists a  Blaschke product  $\varphi$ of  order $4$ such that $\phi\sim \varphi(z^2)$. We prove that $\phi$ is not reducible by contradiction. Otherwise, $\phi\sim \varphi_1\circ \varphi_2$, where $\varphi_1,\varphi_1$ are   Blaschke products of order $2$. This implies that
 $\phi \sim \varphi_1\circ B$  for    a Blaschke product  $B$ of order $4$, which leads to a contraction since by Lemma 4.1 $B^{-1}\circ B$ forms  a subgroup of order $4$ in $\phi^{-1}\circ \phi$, as desired.
 \vskip2mm
  {\bf Subcase (D):}  $G_1=\{0\},G_2=\{4\}$ and $t=3$. Then the partition consists of $G_1,G_2,G_3,G_4$ with $\sharp G_3=\sharp G_4=3$. Considering condition $(\alpha_2)$ and observing that $4$ is the unique element other than $0$ for which its inverse is itself, one sees that $G_4^{-1}=G_3$. The following partitions are all possible choices at this point:
  \begin{enumerate}
\item[(D1)]\,\,\quad $\{ \{0\}, \{4\},\{1,2,3\},\{7,6,5\} \}$;
\item[(D2)]\,\,\quad $\{ \{0\}, \{4\},\{1,2,5\},\{7,6,3\} \}$;
\item[(D3)]\,\,\quad $\{ \{0\}, \{4\},\{1,6,3\},\{7,2,5 \} \}$;
\item[(D4)]\,\,\quad $\{ \{0\}, \{4\},\{1,6,5\},\{7,2,3 \} \}$.
\end{enumerate}
The case  $(D1)$ is impossible by condition $(\alpha_1)$, since   $$\{1,2,3\}+\{7,6,5\}=\{0,7,6,1,0,7,2,1,0\}$$ is not a union of some subsets in $(D1)$. One can prove similarly that $(D2),(D3)$ and $(D4)$ don't satisfy  condition $(\alpha_1)$.
\vskip2mm
 {\bf Subcase (E):}  $G_1=\{0\},G_2=\{4\}$ and $t=2$.

 One possibility is that the partition consists of $G_1,G_2,G_3,G_4$ with $\sharp G_3=2$ and $\sharp G_4=4$. By Condition $(\alpha_2)$, we have   $G_k^{-1}=G_k$ for each $G_k$. So, the only  possibilities are:
  \begin{enumerate}
\item[(E1)]\,\,\quad $\{ \{0\}, \{4\},\{1,7\},\{2,3,5,6\} \}$;
\item[(E2)]\,\,\quad $\{ \{0\}, \{4\},\{2,6\},\{1,3,5,7\} \}$;
\item[(E3)]\,\,\quad $\{ \{0\}, \{4\},\{3,5\},\{1,2,6,7\} \}$.
\end{enumerate}
One excludes   case $(E1)$ by
$$\{4\}+\{1,7\}=\{5,3\},$$
and   case $(E3)$ by
$$\{4\}+\{3,5\}=\{7,1\}.$$
Another possibility is that $\sharp G_k=2$ for any $G_k$ in the partition other than $G_1,G_2$. There exist $C_6^2 C_4^2 C_2^2/A_3^3=15$ choices:
  \begin{enumerate}
\item[(E4)]\,\,\quad $\{ \{0\}, \{4\},\{1,2\},\{3,5\},\{6,7\} \}$; [(E5)]\,\,\quad $\{ \{0\}, \{4\},\{1,2\},\{3,6\},\{5,7\}\}$;
\item[(E6)]\,\,\quad $\{ \{0\}, \{4\},\{1,2\},\{3,7\},\{5,6\} \}$; [(E7)]\,\,\quad $\{ \{0\}, \{4\},\{1,3\},\{2,5\},\{6,7\} \}$;
\item[(E8)]\,\,\quad $\{ \{0\}, \{4\},\{1,3\},\{2,6\},\{5,7\} \}$; [(E9)]\,\,\quad $\{ \{0\}, \{4\},\{1,3\},\{2,7\},\{5,6\} \}$;
\item[(E10)]\,\,\quad $\{ \{0\}, \{4\},\{1,5\},\{2,3\},\{6,7\} \}$; [(E11)]\quad $\{ \{0\}, \{4\},\{1,5\},\{2,6\},\{3,7\} \}$;
\item[(E12)]\,\,\quad $\{ \{0\}, \{4\},\{1,5\},\{2,7\},\{5,6\}\}$; [(E13)]\quad $\{ \{0\}, \{4\},\{1,6\},\{2,3\},\{5,7\} \}$;
\item[(E14)]\,\,\quad $\{ \{0\}, \{4\},\{1,6\},\{2,5\},\{3,7\} \}$; [(E15)]\quad $\{ \{0\}, \{4\},\{1,6\},\{2,7\},\{3,5\} \}$;
\item[(E16)]\,\,\quad $\{ \{0\}, \{4\},\{1,7\},\{2,3\},\{5,6\} \}$; [(E17)]\quad $\{ \{0\}, \{4\},\{1,7\},\{2,5\},\{3,6\} \}$;
\item[(E18)]\,\,\quad $\{ \{0\}, \{4\},\{1,7\},\{2,6\},\{3,5 \} \}$.
\end{enumerate}
One excludes most of them by the following observation: if $\{a,b\}$ is included in one of the above partitions, then one of the equations $a+b=0$, $a+b=4$ and $a=4+b$ holds. Indeed,
by Condition $(\alpha_1)$, $$ \{a,b\}+\{a,b\}=\{2a,a+b,a+b,2b\}$$
is a union of some $G_k\,'s.$ If $\{a+b\}$ is a singleton, then $a+b=0$ or $a+b=4$. Otherwise, $a+b$ is including
in some $G_k$ satisfying  $\sharp G_k>1$. Noticing that each element of $G_k$ is included in  $ \{a,b\}+\{a,b\}$, one sees that $\sharp G_k\leq 3$. It's easy to verify
that $\sharp G_k\neq 3$ since we assume that the singleton $\{a+b\}$ is not   in the partition. So, $\sharp G_k=2$ and
$$G_k=\{2a,a+b\}=\{a+b,2b\}.$$
That is, $2a=2b$. This means that  $ a=4+b$. Furthermore, noticing that both $2a$ and $a+b=2a+4$ are  even in that case, one sees that $G_k=\{2,6\}$.

By this observation, all the partitions other than $(E8)$,$(E11)$ and $(E18)$ are excluded. By a direct computation, one sees that $(E8)$, $(E11)$ and $(E18)$ satisfy the other conditions, too.

Moreover, the above argument shows that $(E2)$, $(E8)$, $(E11)$  and  $(E18)$  are all the possible partitions that include the sets  $\{0\},\{4\},\{2,6\}$. By Lemma 4.1 and \cite[Theorem 2.1]{SZZ}, there exists a Blaschke product $\psi$ of order $2$ and    a Blaschke product $\varphi$ of order $4$, such that $\phi=\psi\circ \varphi$ and $\varphi$ is included in Case $2$ in  \cite[Theorem 2.1]{SZZ}. This implies that $\phi$ has the desired decomposition.
\vskip2mm

We now turn to the cases $s>1$. Firstly, by Condition $(\alpha_2)$, $4$ is not included  in any $G_k$ for which   $\sharp G_k $ is even. Otherwise, if $4\in G_k$, then $G_k^{-1}=G_k$ since $4$ is the unique element other than $0$ for which its  inverse is itself. Therefore, $$G_k=\{4,k_1,\cdots,k_i,8-k_1,\cdots,8-k_i\}$$
 for some $k_1,\cdots, k_i$. This contradicts  the fact that  $\sharp G_k$ is even. So, $4\notin G_k$ if $\sharp G_k $ is even.

Secondly,  the argument used in analyzing subcase $(E)$ is still valid. Hence, if $\{a,b\}$ is   in the partition, then $a+b=0$ or $a=4+b$. In the latter case,
$\{2,6\}$ is   in the partition. Moreover, since $\{a,b\}+\{a,b\}$ is a union  of some $G_k\,'s$ satisfying $\sharp G_k\leq 2$, and $4$ is not included in any such  $G_k$, we have that $4\neq 2a,2b,2(a+b)$. Therefore,  neither $ 2$ nor $6 $ can be included in any $G_k$ when the partition satisfies  $s>1$ and $\sharp G_k=2$. It also implies that $a+b=0$ if $\{a,b\}$ is   in the partition.

{\bf (II) Case $s=2.$}

One possibility is that the partition consists of $G_1,G_2,G_3$ satisfying   $\sharp G_2=2$ and $\sharp G_3=5$. By the above observation, such  partition is one of the following:
  \begin{enumerate}
\item[(II1)]\,\,\quad $\{ \{0\}, \{1, 7\},\{2,3,4,5,6\} \}$;
\item[(II2)]\,\,\quad $\{ \{0\}, \{3, 5\},\{1,2,4,6,7\} \}$.
\end{enumerate}
Obviously, none of them satisfies Condition $(\alpha_1)$.

Another scenario is that the partition consists of $G_1,G_2,G_3, G_4$ satisfying  $\sharp G_2=\sharp G_3=2$ and $\sharp G_4=3$. By the above argument,  $G_4=\{2,4,6\}$. So, all the possibilities are listed below:
  \begin{enumerate}
\item[(II3)]\,\,\quad $\{ \{0\},\{1,3\},\{5,7\}, \{2,4,6\}\} $;
\item[(II4)]\,\,\quad $\{ \{0\}, \{1,5\},\{3,7\},\{2,4,6\}\} $;
\item[(II5)]\,\,\quad $\{ \{0\},\{1,7\},\{3,5\}, \{2,4,6\}\} $.
\end{enumerate}
None of them satisfies Condition $(\alpha_1)$.\vskip2mm
{\bf (III) Case $s=3$.}

In this case, the partition consists of $G_1,G_2,G_3$ satisfying $\sharp G_2=3$ and $\sharp G_3=4$. By the above argument and Condition $(\alpha_2)$, one sees that $G_2^{-1}=G_2$, $G_3^{-1}=G_3$ and $4\in G_2$. So, the partition is one of the following:
  \begin{enumerate}
\item[(III1)]\,\,\quad $\{ \{0\}, \{1,4,7\},\{2,3,5,6\} \}$;
\item[(III2)]\,\,\quad $\{ \{0\}, \{2,4,6\},\{1,3,5,7\} \}$;
\item[(III3)]\,\,\quad $\{ \{0\}, \{3,4,5\},\{1,2,6,7\} \}$.
\end{enumerate}
Both $\mathrm{(III1)}$ and $\mathrm{(III2)}$ are excluded  by Condition $(\alpha_1)$, since
$\{1,4,7\}+\{1,4,7\}$ and $\{3,4,5\}+\{3,4,5\}$ are not unions of some subsets in the partitions, respectively. For the finial possibility  $\{ \{0\}, \{2,4,6\},\{1,3,5,7\} \}$, using an argument similar  to the  above, one sees that it is equivalent  to the condition that    $\phi\sim  \psi\circ  \varphi  $, where $\psi$ is a  Blaschke product of order $2$ and $\varphi$ is a Blaschke product of order $4$, and $\varphi$ is included in case $3$ in \cite[Theorem 2.1]{SZZ}.
\vskip2mm
{\bf (IV) Case $s=7$.}

The only choice is $\{ \{0\}, \{1,2,3,4,5,6,7\}\}$. By Lemma 4.1, $\phi$ is not reducible in this case. $\quad\quad\quad\quad\quad\quad\quad\quad\quad\quad\quad\quad\quad\quad\quad\quad\quad\quad\quad\quad\quad\quad\quad\quad\quad\quad\quad\quad\quad\quad\quad\quad \quad\quad\Box$
\vskip3mm
We conclude with the following corollary which follows after one summarizes  all the  possibilities listed above.
\begin{cor}
Let $\phi$ be a finite Blaschke product of order 8. Then $M_\phi$ has exactly $2$ nontrivial minimal reducing subspaces if and only if $\phi$ is not reducible.
\end{cor}
It is natural to ask if this result   extends to the general case. One can obtain  a similar result for order $6$ by the above arithmetic way. But, the calculation
for order $5$ or $7$ suggests that some  counterexample may exist, although we can't exhibit it. A possible    guess may be that the result holds whenever the order of $\phi$ is not prime.
 \bibliographystyle{plain}

\end{document}